\documentclass[12pt,reqno]{amsart}
\usepackage{amsmath, amsfonts, amssymb, amsthm}
\def\Z{\mathbb Z}

\def\N{\mathbb N}
\def\Q{\mathbb Q}
\def\R{\mathbb R}

\def\cB{\mathcal B}
\def\cJ{\mathcal J}

\def\cP{\mathcal P}
\def\cA{\mathcal A}

\def\cI{\mathcal I}

\def\fd{\mathfrak d}

\def\1{{\bf 1}}

\def\pmod #1{\ ({\rm{mod}}\ #1)}

\theoremstyle{plain}
\newtheorem{theorem}{Theorem}
\newtheorem{proposition}{Proposition}
\newtheorem{lemma}{Lemma}

\theoremstyle{definition}

\theoremstyle{remark}

\begin{document}
\title{The Green-Tao theorem for  Piatetski-Shapiro primes}

\author{Hongze Li}
\address{School of Mathematical Sciences, Shanghai Jiao Tong University, Shanghai 200240, People's Republic of China}
\email{lihz@sjtu.edu.cn}
\author{Hao Pan}
\address{School of Applied Mathematics, Nanjing University of Finance and Economics, Nanjing 210046, People's Republic of China}
\email{haopan79@zoho.com}
\begin{abstract}
Let $m\geq 3$.
Suppose that
$$
1-2^{-2^{m^24^m}}<\gamma<1.
$$
Then the set
$$
\{p\text{ prime}:\, p=[n^{\frac1\gamma}]\text{ for some }n\in\N\}
$$
contains infinitely many non-trivial $m$-term arithmetic progressions.

\end{abstract}
\maketitle

\section{Introduction}
\setcounter{lemma}{0}
\setcounter{theorem}{0}
\setcounter{corollary}{0}
\setcounter{equation}{0}
\setcounter{conjecture}{0}

The well-known Green-Tao theorem \cite{GT08} asserts that for each $m\geq 3$, the set of all primes $\cP$ contains infintely many non-trivial arithmetic progressions of length $m$. That is, for each $m\geq 3$, there exists infinitely many pairs of positive integers $a,d$ such that $a,a+d,\ldots,a+(m-1)d$ are all primes.  In fact, Green and Tao proved a Szemer\'edi-type theorem for primes: any subset $A$ of $\cP$ with $\overline{d}_\cP(\cA)>0$ contains infinitely many non-trivial $m$-terms arithmetic progressions for each $m\geq 3$, where the relative upper density
$$
\overline{d}_\cP(\cA):=\limsup_{X\to\infty}\frac{|\cA\cap[1,X]|}{|\cP\cap[1,X]|}.
$$

The Green-Tao theorem has been generalized for the primes of some special forms, including the Chen primes (by Zhou \cite{Zhou09}), the primes $p$ such that the interval $[p+1,p+7\times 10^7]$ contains at least one prime (by Pintz \cite{Pintz16}), the primes of the form $x^2+y^2+1$ (by Sun and Pan \cite{SP}), etc.. For the further generalizations of the  Green-Tao theorem, the reader may refer to
\cite{GT10, Le11, LW14, Tao06, TZ08,TZ015}.

The Piatetski-Shapiro prime is another kind of primes of the special form.
A well-known conjecture asserts that there exist infinitely many primes of the form $n^2+1$. This conjecture is far from solved under the current techniques, though Iwaniec \cite{Iw78} proved that there exist infinitely many $n$ such that $n^2+1$ has at most two prime factors.
In 1953, Piatetski-Shapiro \cite{PS53} considered another approximation to this conjecture.
Suppose that $\gamma\in(0,1)$ and $\gamma^{-1}\not\in\Z$. Then
$x^{\frac1\gamma}$ can be viewed as a polynomial of degree $\gamma^{-1}$.
Let
$$
\N^{\frac1\gamma}:=\big\{[n^\frac1\gamma]:\,n\in\N\big\}
$$
where $[x]$ denotes the integer part of $x$,
i.e., $\N^{\frac1\gamma}$ is the set of all non-negative integers of the form $[n^{\frac1\gamma}]$. Clearly $|\N^{\frac1\gamma}\cap[0,x]|=x^{\gamma}+O(1)$ for any sufficiently large $x$.
Piatetski-Shapiro \cite{PS53} proved that for any $\gamma\in(11/12,1)$, there exist infinitely many primes lying in $\N^{\frac1\gamma}$. Explicitly, he obtained that
\begin{equation}\label{ShapiroPNT}
|\cP_\gamma\cap[1,x]|=(1+o(1))\cdot\frac{x^\gamma}{\log x}
\end{equation}
as $x\to +\infty$, where
$$
\cP_{\gamma}:=\{p\text{ prime}:\,p\in\N^{\frac1\gamma}\},
$$
is the set of those primes of the form $[n^{\frac1\gamma}]$. Nowadays, the primes in $\cP_{\gamma}$ is also called {\it Piatetski-Shapiro primes}.  The result of Piatetski-Shapiro has been improved many times. The current best known result on the distribution of the Piatetski-Shapiro primes is due to Rivat and Wu \cite{RW01}, they proved that
\begin{equation}
|\cP_\gamma\cap[1,x]|\gg\frac{x^\gamma}{\log x}
\end{equation}
for any $\gamma\in(205/243,1)$.

On the other hand, with the help of Heath-Brown identity \cite{He83}, Balog and Friedlander \cite{BF92} proved that for any $\gamma\in(20/21,1)$, every sufficiently large odd integer can be represented as the sum of three primes lying in $\cP_\gamma$.
The key ingredient of Balog and Friedlander's proof is the following estimation of exponential sum:
\begin{equation}
\frac1\gamma\sum_{p\in\cP_\gamma\cap[1,x]}p^{1-\gamma}\log p\cdot e(p\theta)= \sum_{p\in[1,x]}\log p\cdot e(p\theta)+O(x^{1-\epsilon})
\end{equation}
for any $\gamma\in(8/9,1)$, where $e(\theta)=\exp(2\pi\sqrt{-1}\theta)$ and $\epsilon>0$ is a constant only depending on $\gamma$. Clearly using the discussions of Balog and Friedlander, one can easily prove that $\cP_\gamma$ contains infinitely many non-trivial three-term arithmetic progressions for each $\gamma\in(20/21,1)$. Furthermore, with the help of the transference principle, Mirek \cite{Mi15} obtained a Roth-type theorem for the Piatetski-Shapiro primes and showed that for any $\gamma\in(71/72,1)$ and any subset $\cA\subseteq\cP_\gamma$ with
$$
\limsup_{x\to+\infty}\frac{|\cA\cap[1,x]|}{x^\gamma(\log x)^{-1}}>0,
$$
$\cA$ contains infinitely many non-trivial three-term arithmetic progressions.

It is natural to ask whether the Piatetski-Shapiro primes contain longer non-trivial arithmetic progressions. In this paper, we shall prove the following Szemer\'edi-type result for Piatetski-Shapiro primes.
\begin{theorem}
\label{main}
Let $m\geq 3$.
Suppose that
\begin{equation}\label{lowerboundgamma}
1-2^{-2^{m^24^m}}<\gamma<1
\end{equation}
and $\cA$ is a subset of $\cP_\gamma$ with
$$
\limsup_{x\to+\infty}\frac{|\cA\cap[1,x]|}{x^\gamma(\log x)^{-1}}>0.
$$
Then $\cA$
contains infinitely many non-trivial $m$-term arithmetic progressions.
\end{theorem}
Of course, the lower bound of $\gamma$ in (\ref{lowerboundgamma}) is not very optimal, and surely could be improved via some more accurate calculations. However, we believe that with the help of the current techniques, it is impossible to obtain a lower bound of $\gamma$ independent on $m$. In fact, by the heuristic discussions, (\ref{ShapiroPNT}) should be valid for any $\gamma\in(0,1)$ with $\gamma^{-1}\not\in\Z$, which is evidently very far from being proved. On the other hand, Green and Tao \cite{GT10} introduced the complexity of linear equations and showed that the non-trivial $m$-term arithmetic progressions factly correspond to a linear equation with the complexity $m-2$. Now for the simplest equation $p=[n^{\frac1\gamma}]$, we only can prove the existence of solutions when $205/243<\gamma<1$. So for those equations with higher complexity, the lower bound of $\gamma$ would naturally become worse.

Let us introduce the strategy for the proof of Theorem \ref{main}. The key ingredient is to construct a suitable pseudorandom measure for Piatetski-Shapiro primes and verify the corresponding linear forms condition. By revising the construction of Green-Tao, it is easy to obtain such a  pseudorandom measure. However, our main difficulty is how to verify the linear forms condition. One reason is that the Piatetski-Shapiro prime theorem  (\ref{ShapiroPNT}) arises from the estimations of exponential sums, rather than the sieve method. Explicitly, in order to prove Theorem \ref{main}, we have to give a non-trivial upper bound for the exponential sum of the form
\begin{equation}\label{expsumsipsii}
\sum_{X\leq n\leq X+Y}e\big(s_1\psi_1(n)+s_2\psi_2(n)+\cdots+s_h\psi_h(n)\big),
\end{equation}
where $s_1,\ldots,s_h\in\Z$ and $\psi_1,\ldots,\psi_h$ are some linear functions.

The classical van der Corput theorem is a useful tools to estimate the exponential sum of the form
$$
\sum_{X\leq n\leq X+Y}e\big(f(n)\big),
$$
where $f$ is a smooth function over the interval $[X,X+Y]$. Suppose that
\begin{equation}\label{f2derlambda}
\lambda\leq |f''(x)|\leq \alpha\lambda
\end{equation}
for each $x\in[X,X+Y]$, where $\alpha, \lambda>0$ are independent on $x$. Then the van der Corput theorem asserts that
$$
\sum_{X\leq n\leq X+Y}e\big(f(n)\big)\ll
\alpha Y(\lambda^{\frac1{2}}
+Y^{-1}\lambda^{-\frac1{2}}).
$$
Unfortunately, since it is possible that some of $s_i$ in (\ref{expsumsipsii}) are negative and the others are positive,
for the function $f(x)=s_1\psi_1(x)+\cdots+s_h\psi_h(x)$, generally (\ref{f2derlambda}) doesn't hold.

Our strategy is to apply the generalized van der Corput theorem concerning the derivatives of higher order. However, for $f(x)=\sum_{i}s_i\psi_i(x)$ and any given integer $r\geq 2$, we also don't know whether $|f^{(r)}(x)|$ could be bounded by $\lambda$ and $\alpha\lambda$. So there are two key ingredients in our proof.
First, with the help of some suitable linear transformations, the estimation of (\ref{expsumsipsii}) can be reduced to a special case that $\psi_i(x)=a_ix+b_i$ with $1\leq |a_1|<|a_2|<\ldots<|a_h|\leq M$.
Next,  we can show that for any $s_1,\ldots,s_h\in\Z$, there exists $r\in[2,c(M)]$, where $c(M)$ is a constant only depending on $M$, such that
$$
\lambda\leq |f^{(r)}(x)|\leq \alpha\lambda
$$
for any $x\in[X,X+Y]$, where $\alpha,\lambda>0$ and $\alpha$ only depends on $M$. Thus by using a generalization of van der Corput's theorem, we can get a desired upper bound for the exponential sum (\ref{expsumsipsii}).

The whole paper will be organized as follows. First, in the next section, we shall introduce Green-Tao's transference principle and give the definitions of the pseudorandom measure and the linear forms condition. Then in the third section, we shall construct a pseudorandom measure $\nu$ for Piatetski-Shapiro primes. In Section 4, in order to verify the linear forms condition for $\nu$, we shall reduce a
 Goldston-Y{\i}ld{\i}r{\i}m-type estimation to the estimation of some exponential sums. Section 5 is the core part of the proof of Theorem \ref{main}, which will contain two key auxiliary lemmas. Finally, in Section 6, we shall complete the proof
 of Theorem \ref{main} by combining a generalized van der Corput theorem with the two lemmas in Section 5.

Throughout this paper,
$f(x)\ll g(x)$ means $f(x)=O\big(g(x)\big)$ as $x$ tends to $\infty$.
Furthermore, without the additional mentions, the implied constants in $O$, $\ll$ and $\gg$ at most depends on $m$.
As usual, let $\phi(\cdot)$ and $\mu(\cdot)$ denote the Euler totient function and the M\"obius  function respectively. Furthermore, let $\log_k$ denote the $k$-th iteration of the logarithm function.

\section{Transference Principle}
\setcounter{lemma}{0}
\setcounter{theorem}{0}
\setcounter{corollary}{0}
\setcounter{equation}{0}
\setcounter{conjecture}{0}

Let $N$ be a sufficiently large prime and let  $\Z_N:=\Z/N\Z$ denote the cyclic group of order $N$. Suppose that $$\nu(n):\,\Z_N\to\R$$ is a non-negative function.
First, we introduce the definition of the $(h_0,k_0,m_0)$-linear forms condition. For $1\leq h\leq h_0$ and $1\leq k\leq k_0$, suppose that
$$
\psi_i(x_1,\ldots,x_k)=a_{i1}x_1+\cdots+a_{ik}x_k+b_i,\qquad 1\leq i\leq h,
$$
where $b_i$, $a_{ij}\in\Q$  and the numerators and denominators of those $a_{ij}$ all lie in $[-m_0,m_0]$.
Furthermore, assume that $\psi_i(x_1,\ldots,x_k)$ is not a rational multiple of $\psi_j(x_1,\ldots,x_k)$ for any distinct $i,j$.
Clearly, we may also view those $\psi_i$ as the linear functions over $\Z_N$ whenever $N>m_0$.
Then we say $\nu$ obeys the {\it $(h_0,k_0,m_0)$-linear forms condition}, provided that as $N\to\infty$,
\begin{equation}
\frac{1}{N^k}\sum_{x_1,\ldots,x_k\in\Z_N}\nu\big(\psi_1(x_1,\ldots,x_k)\big)\cdots
\nu\big(\psi_h(x_1,\ldots,x_k)\big)=1+o_{h_0,k_0,m_0}(1),
\end{equation}
for any $1\leq h\leq h_0$, $1\leq k\leq k_0$ and those $\psi_1,\ldots,\psi_h$.

Next, we say $\nu$ satisfies the {\it $k_0$-correlation condition}, if for any $1\leq k\leq k_0$, there exists a non-negative weight function $\tau_k:\,\Z_N\to\R$ such that
$$
\frac1N\sum_{x\in\Z_N} \tau_k(x)^s=O_{k,s}(1)
$$
for any integer $s\geq 1$, and
\begin{equation}
\frac1N\sum_{x\in\Z_N} \nu(x+b_1)\cdots\nu(x+b_k)\leq \sum_{1\leq i<j\leq k}\tau_k(b_i-b_j)
\end{equation}
for any $b_1,\ldots,b_k\in\Z_N$.

Call $\nu(x)$ a {\it $m$-pseudorandom measure}, provided that $\nu$ obeys $(2^{m-1}m,3m-4,m)$-linear forms condition and $2^{m-1}$-correlation condition.
The important transference principle of Green and Tao \cite[Theorem 3.5]{GT08} asserts that
\begin{theorem}
Suppose that $\delta>0$ and $m\geq 3$. Let $f(x)$ be a function over $\Z_N$ such that
$$
\frac1N\sum_{x\in\Z_N}f(x)\geq\delta,
$$
and
$$
0\leq f(x)\leq \nu(x)
$$
for each $x\in\Z_N$, where $\nu$ is a $m$-pseudorandom measure over $\Z_N$. Then as $N\to\infty$,
\begin{equation}\label{fkaa}
\frac1{N^2}\sum_{x,y\in\Z_N}f(x)f(x+y)\cdots f\big(x+(m-1)y\big)\geq c(m,\delta)+o_{k,\delta}(1),
\end{equation}
where $c(m,\delta)$ is a constant only depending on $m$ and $\delta$.
\end{theorem}
With the help of the arguments of Goldston and
Y\i ld\i r\i m \cite{GY03}, for any $m\geq 3$, Green and Tao constructed a $m$-pseudorandom measure $\nu(n)$ over $\Z_N$ such that
\begin{equation}\label{nuc0lambda}
\nu(n)\geq c_0\cdot\lambda_{W,b}(n)
\end{equation}
for any $n\in[\epsilon_0N,2\epsilon_0N]$, where $W=\prod_{p\leq\log_4N}p$,
\begin{equation}\label{lambdaWbn}
\lambda_{W,b}(n)=\begin{cases}\frac{\phi(W)}W\cdot\log(Wn+b),&\text{if }Wn+b\text{ is prime},\\
0,&\text{otherwise},
\end{cases}
\end{equation}
and $c_0,\epsilon_0>0$ are two small constants only depending on $m$.
Let
$$
f(x)=\begin{cases}c_0\lambda_{W,b}(x),&\text{if }x\in[\epsilon_0N,2\epsilon_0N],\\
0,&\text{otherwise},
\end{cases}
$$
be a function over $\Z_N$.
According to the Siegel-Walfisz theorem, we have
$$
\sum_{x\in\Z_N}f(x)=c_0\epsilon_0N+o(N)
$$
as $N\to\infty$.
It follows from (\ref{fkaa}) and (\ref{nuc0lambda}) that
\begin{equation}\label{xrfcN2}
\sum_{x,r\in\Z_N}f(x)f(x+r)\cdots f\big(x+(m-1)r\big)\geq c_1N^2
\end{equation}
for some constant $c_1>0$. Clearly $f(x)f(x+r)\cdots f(x+(m-1)r)>0$ implies that $
x,x+r,\ldots,x+(m-1)r$ modulo $N$ all lie in the interval $[\epsilon_0N,2\epsilon_0N]$.
Since $\epsilon_0<1/2$ and $x\in [\epsilon_0N,2\epsilon_0N]$, if $1\leq r<N/2$, it is impossible that $x+r-N\in [\epsilon_0N,2\epsilon_0N]$. Hence we must have $x+r\in [\epsilon_0N,2\epsilon_0N]$, as well as $x+2r,\ldots,x+(m-1)r$. Suppose that $N/2<r<N$. Letting $r'=N-r$ and $x'\in [\epsilon_0N,2\epsilon_0N]$ with $x'\equiv x+(m-1)r\pmod{N}$, we also have $x'+r',\ldots,x'+(m-1)r'\in[\epsilon_0N,2\epsilon_0N]$.
Recall that $f(n)\leq 2c_0\phi(W)\log N/W$. We obtain that there exist at least
$$
\frac{c_1N^2}{2}\cdot\bigg(\frac{W}{2c_0\phi(W)\log N}\bigg)^m-\epsilon_0N
$$
pairs of $(x,r)$ with $1\leq x,r<N$ such that $Wx+b,\ldots,W(x+(m-1)r)+b$ form a non-trivial arithmetic progression in primes..

Green and Tao's transference principle becomes a powerful tools to prove the relative Szemer\'edi-type theorems nowadays.
Furthermore, in \cite{CFZ15}, Conlon, Fox and Zhao weakened the requirements concerning the pseudorandom measures. In fact, they defined the notion of $k$-linear forms condition. Let $\nu$ be a non-negative function over $\Z_N$. Suppose that as $N\to\infty$,
$$
\frac1{N^{2k}}\sum_{x_1,y_1,\ldots,x_k,y_k\in\Z_N}\prod_{j=1}^k\prod_{\substack{
I,J\subseteq\{1,\ldots,k\}\\
I\cup J=\{1,\ldots,k\}\setminus\{j\}\\ I\cap J=\emptyset
}}\nu\bigg(\sum_{s\in I}(s-j)x_s+\sum_{t\in J}(t-j)y_t\bigg)^{\delta_{j,I,J}}=1+o(1)
$$
for any choice of $\delta_{j,I,J}\in\{0,1\}$. Then we say $\nu$ obeys the {\it $k$-linear forms condition}.
Conlon, Fox and Zhao proved that
\begin{theorem}\label{CFZThm}
Suppose that $\delta>0$ and $m\geq 3$. Let $f(x)$ be a non-negative function over $\Z_N$ such that
$$
\frac1N\sum_{x\in\Z_N}f(x)\geq\delta,
$$
and
$$
0\leq f(x)\leq \nu(x)
$$
for some function $\nu$ obeys the $m$-linear forms condition. Then (\ref{fkaa}) is also valid.
\end{theorem}
Clearly the $m$-linear forms condition is weaker than
the $(2^{m-1}m,2m,m)$-linear forms condition.
So in order to get a relative Szemer\'edi-type theorem for the Piatetski-Shapiro primes,  we only need to construct a suitable pseudorandom measure $\nu$ over $\Z_N$ and verify the $(2^{m-1}m,2m,m)$-linear forms condition for $\nu$. This is our main task in the remainder sections.

\section{Pseudorandom Measure}
\setcounter{lemma}{0}
\setcounter{theorem}{0}
\setcounter{corollary}{0}
\setcounter{equation}{0}
\setcounter{conjecture}{0}

Let $\gamma$ and $m$ be given in Theorem \ref{main}.
In this section, we shall construct a pseudorandom measure $\nu$ for those primes in $\cP_\gamma$.
Let
$$
\kappa_1=\frac{1}{2^{m}\cdot(m+4)!},\qquad
\kappa_2=\kappa_1\cdot\bigg(1+\frac{1}{2^{4m-2}m^5}\bigg).
$$
Suppose that $X$ is a sufficiently large positive integer. Let
$$
W=\prod_{\substack{p\leq\log_4X\\ p\text{ is prime}}}p.
$$
According to the Piatetski-Shapiro prime number theorem,
$$
\big|\cP_\gamma\cap[(\kappa_1+\kappa_0)X,(\kappa_2-\kappa_0)X]\big|=\big(\kappa_2-\kappa_1-2\kappa_0+o(1)\big)\cdot\frac{X^{\gamma}}{\log X},
$$
where $\kappa_0=(\kappa_2-\kappa_1)/100$.
Recall that $\cA$ is a subset of $\cP_\gamma$ with a positive relatively upper density. Let
$$
\fd_0=\limsup_{x\to+\infty}\frac{|\cA\cap [1,x]|}{|\cP_\gamma\cap[1,x]|}
$$
be the relatively upper density of $\cA$.
Then we may choose a sufficiently large $X$ such that
$$
\big|\cA\cap[(\kappa_1+\kappa_0)X,(\kappa_2-\kappa_0)X\big|\geq
\frac{4\fd_0(\kappa_2-\kappa_1)}{5}\cdot\frac{X^{\gamma}}{\log X}.
$$

Let $N$ be a prime lying in $[(1-\kappa_0)X/W,(1+\kappa_0)X/W]$. By the prime number theorem, such prime $N$ always exists. Clearly $$[\kappa_1WN,\kappa_2WN]\supseteq [(\kappa_1+\kappa_0)X,(\kappa_2-\kappa_0)X].$$
By the pigeonhole principle, there exists $1\leq b_0\leq W$ with $(b_0,W)=1$ such that
\begin{equation}\label{WnbA}
\big|\{Wn+b_0\in\cA:\, \kappa_1N\leq n\leq \kappa_2N\}\big|\geq\frac1{\phi(W)}\cdot\frac{\fd_0(\kappa_2-\kappa_1)}{2}\cdot\frac{(WN)^{\gamma}}{\log N}.
\end{equation}

Let
$$
\eta_0=2\gamma\cdot (\kappa_1WN)^{\gamma-1}.
$$
For any $n\in[\kappa_1N,\kappa_2N]$, we have
$$
(Wn+b_0+1)^{\gamma}-(Wn+b_0)^\gamma\leq \gamma\cdot (Wn+b_0)^{\gamma-1}\leq\eta_0.
$$
On the other hand, clearly $Wn+b_0\in\N^{\frac1\gamma}$ if and only if $$
\big[(Wn+b_0+1)^{\gamma}\big]>\big[(Wn+b_0)^{\gamma}\big]
\qquad\text{or}\qquad (Wn+b_0)^{\gamma}\in\N.$$
Clearly $\big[(Wn+b_0+1)^{\gamma}\big]>\big[(Wn+b_0)^{\gamma}\big]$
is equivalent to $$
(Wn+b_0+1)^{\gamma}-(Wn+b_0)^{\gamma}\geq 1-\big\{(Wn+b_0)^{\gamma}\big\}.
$$
So if $Wn+b_0\in\N^{\frac1\gamma}$, then
$$
1-\big\{(Wn+b_0)^\gamma\big\}\leq\eta_0
\qquad\text{or}\qquad
\big\{(Wn+b_0)^\gamma\big\}=0.
$$

Let
$$
h_0=2^{m-1}m,\qquad
k_0=2m,
$$

$$
\delta_0=\frac1{3}\bigg(\gamma+\frac{1}{2^{2^{m^24^m}}}-1\bigg),
$$
and $r_0$ be the least positive integer such that
$$
\bigg(r_0-\frac{h+2}{1-\gamma+\delta_0}\bigg)\cdot(1-\gamma+\delta_0)> r_0(1-\gamma)
$$
for each $1\leq h\leq h_0$.
According to \cite[Lemma 12 of Chapter I]{Vi54}, there exists a smooth function $\rho(t)$ with the period $1$ such that

\medskip
\noindent(i) $0\leq \rho(t)\leq 1$ for any $t$ and
\begin{equation}
\rho(t)=\begin{cases}
1, &\text{if }1-\eta_0\leq t\leq 1,\\
0, &\text{if }\eta_0\leq t\leq 1-2\eta_0,
\end{cases}
\end{equation}

\noindent(ii)
\begin{equation}\label{rhoalphasum}
\rho(t)=2\eta_0+\sum_{|j|\geq 1}\alpha_je(jt),
\end{equation}
where
\begin{equation}\label{fourieralphabound}
\alpha_j\ll_{r_0}\min\bigg\{\eta_0,\ \frac1{|j|},\ \frac1{\eta_0^{r_0}|j|^{r_0+1}}\bigg\}.
\end{equation}

\medskip\noindent
Thus for any $n\in[\kappa_1N,\kappa_2N]$, $Wn+b_0\in\N^{\frac1\gamma}$ implies that $$\rho\big((Wn+b_0)^\gamma\big)=1.$$

Let
$$
R=N^{\delta_0}.
$$
Define
\begin{equation}\label{LambdaRn}
\Lambda_R(n):=\sum_{\substack{d\mid n\\ d\leq R}}\mu(d)\log\frac{R}{d}.
\end{equation}
Clearly if $n>R$ is a prime, then
$\Lambda_R(n)=\log R$.
Let
\begin{equation}
\varpi(n)=\frac{\Lambda_R(n)^2}{\log R}\cdot \frac{\rho(n^\gamma)}{2\eta_0},
\end{equation}
and define
\begin{equation}\label{nufunc}
\nu(n):=\begin{cases}\frac{\phi(W)}W\cdot\varpi(Wn+b_0),&\text{if }n\in[\kappa_1N,\kappa_2N],\\
1,&\text{otherwise}.
\end{cases}
\end{equation}
Set
\begin{equation}\label{primefunc}
f_0(n)=\begin{cases} \frac{\delta_0}{3\eta_0}\cdot\lambda_{W,b_0}(n),&\text{if }n\in[\kappa_1N,\kappa_2N]\text{ and }Wn+b_0\in\cA,\\
0,&\text{otherwise},
\end{cases}
\end{equation}
where $\lambda_{W,b_0}$ is the one given in (\ref{lambdaWbn}).
If $n\in[\kappa_1N,\kappa_2N]$ and $Wn+b_0\in\cP_{\gamma}$, then
\begin{equation}\label{f0nu}
f_0(n)\leq \frac{\delta_0}{3\eta_0}\cdot\frac{\phi(W)}{W}\cdot\log(\kappa_2 WN+W)\leq \frac{\phi(W)}{2\eta_0W}\cdot\log R=\nu(n).
\end{equation}
That is, we always have
$$
0\leq f_0(n)\leq\nu(n)
$$
for every $1\leq n\leq N$.

In view of (\ref{WnbA}),
\begin{align*}
\sum_{\substack{n\in[\kappa_1N,\kappa_2N]\\ Wn+b_0\in\cA}}f_0(n)\geq
\frac{\fd_0(\kappa_2-\kappa_1)}{2\phi(W)}\cdot\frac{(WN)^{\gamma}}{\log N}\cdot \frac{\delta_0}{3\eta_0}\cdot\frac{\phi(W)\log(\kappa_1 WN)}{W}\geq\frac{\fd_0(\kappa_2-\kappa_1)}{13\gamma\kappa_1^{\gamma-1}}\cdot N,
\end{align*}
by recalling that $\eta_0=2\gamma\cdot (\kappa_1WN)^{\gamma-1}$. Hence by Theorem \ref{CFZThm}, if $\nu$ obeys the $m$-linear forms condition, then
$$
\sum_{x,r\in[1,N]}f_0(x)f_0(x+r)\cdots f_0\big(x+(m-1)r\big)\geq c_{m,\fd_0} N^2
$$
for some constant $c_{m,\fd_0}>0$ only depending on $m$ and $\fd_0$. By (\ref{f0nu}), we have $f_0(x)\leq \log R\cdot \phi(W)/(2\eta_0W)$.
According to the discussions after (\ref{xrfcN2}), there exist at least
$$
\frac{c_{m,\fd_0}N^2}2\cdot\bigg(\frac{2\eta_0W}{\phi(W)\log R}\bigg)^m-(\kappa_2-\kappa_1)N
$$ pairs of $(x,r)$ with $1\leq x,r<N$ such that $Wx+b_0,\ldots,W(x+(m-1)r)+b_0$ form a non-trivial arithmetic progression in $\{p\in\cA:\, p\equiv b_0\pmod{W}\}$. Thus Theorem \ref{main} is concluded.

Our remainder task is to verify the $(2^{m-1}m,2m,m)$-linear forms condition
for the measure $\nu$. In the next section, we shall propose a Goldston-Y{\i}ld{\i}r{\i}m-type estimation for $\nu$, which evidently implies the $(2^{m-1}m,2m,m)$-linear forms condition.

\section{The Goldston-Y{\i}ld{\i}r{\i}m-type estimation}
\setcounter{lemma}{0}
\setcounter{theorem}{0}
\setcounter{corollary}{0}
\setcounter{remark}{0}
\setcounter{equation}{0}
\setcounter{conjecture}{0}

Suppose that $1\leq h\leq h_0$ and $1\leq k\leq k_0$. Let $$
\psi_i(x_1,\ldots,x_k)=a_{i1}x_1+\cdots+a_{ik}x_k,\qquad 1\leq i\leq h,
$$
where $a_{ij}\in\Z$ and $|a_{ij}|\leq m$. Further, suppose that
$(a_{i1},\ldots,a_{ik})$ and  $(a_{j1},\ldots,a_{jk})$ are linearly independent for any $1\leq i<j\leq h$.
Below, for convenience, we write $\vec{x}=(x_1,\ldots,x_k)$. Then
we have the following Goldston-Y{\i}ld{\i}r{\i}m-type estimation.
\begin{proposition}\label{linearcondition}
\begin{equation} \label{linearconditione}
\frac1{N^k}\sum_{\vec{x}\in\Z_N^k}\prod_{i=1}^h\nu\big(\psi_i(\vec{x})+b_i\big)=1+o(1),
\end{equation}
for any $b_1,\ldots,b_h\in\Z_N$.
\end{proposition}
In this section, we need to reduce the proof of Proposition \ref{linearcondition} to the estimations of some exponential sums.
We shall follow the the same way of Green and Tao in \cite{GT08}.
Let $Q=[N/\log_4 N]$ and $U=[N/Q]$.
Let
$$
B_{u_1,\ldots,u_k}=\{(x_1,\ldots,x_k):\,u_iQ<x_i\leq (u_i+1)Q\text{ for each }1\leq i\leq k\}
$$
for each $0\leq u_1,\ldots,u_k\leq U-1$, and let $\cB$ be the set of all those $B_{u_1,\ldots,u_k}$. For any $B\in\cB$, we say $B$ is {\it good} provided that for any $1\leq i\leq h$, either
$\psi_i(B)\subseteq[\kappa_1N,\kappa_2N]$ or $\psi_i(B)\cap[\kappa_1N,\kappa_2N]=\emptyset$. Also, we call $B\in\cB$ {\it bad} if $B$ is not good.

Suppose that $B\in\cB$ is good. Let $\cJ_B=\{1\leq j\leq h:\,\psi_j(B)\subseteq[\kappa_1N,\kappa_2N]\}$. In view of (\ref{nufunc}),
\begin{equation}\label{goodBsum}
\sum_{\vec{x}\in B}\prod_{i=1}^h\nu\big(\psi_i(\vec{x}+b_i)\big)=
\frac{\phi(W)^{|\cJ_B|}}{W^{|\cJ_B|}}\sum_{\vec{x}\in B}\prod_{j\in\cJ_B}\varpi\big((\psi_j(\vec{x})+b_j)W+b_0\big).
\end{equation}
On the other hand, according to Green and Tao's discussions in \cite[Page 528]{GT08}, the number of all bad $B\in\cB$ is $O(U^{k-1})$. Hence
\begin{align*}
&\sum_{B\text{ is }bad}\sum_{\vec{x}\in B}\prod_{i=1}^h\nu\big(\psi_i(\vec{x})+b_i\big)+
\sum_{\vec{x}\in \Z_N^{k}\setminus(\bigcup_{B\in\cB}B)}\prod_{i=1}^h\nu\big(\psi_i(\vec{x}+b_i)\big)\\
=&O\bigg(U^{k-1}\max_{\substack{\cI_1,\ldots,\cI_k\subseteq[1,N]\\ |\cI_i|=Q+O(1)}}\sum_{\vec{x}\in \cI_1\times\cdots\times\cI_k}\prod_{i=1}^h\nu\big(\psi_i(\vec{x})\big)\bigg).
\end{align*}
Note that
$$
\nu(n)\leq 1+\frac{\phi(W)}{W}\cdot\varpi(Wn+b_0)
$$
for any $n\in\Z_N$. Therefore, according to (\ref{goodBsum}), Proposition \ref{linearcondition} immediately follows from the estimation
\begin{equation}\label{phiWkWkQkpsixJ}
\frac{\phi(W)^{|J|}}{W^{|J|}Q^k}\sum_{\vec{x}\in\cI_1\times\cdots\times\cI_k}\prod_{j\in J}\varpi\big(\psi_j(\vec{x})W+b_j^*\big)=1+o(1)
\end{equation}
for any $J\subseteq\{1,\ldots,h\}$ and any $\cI_1,\ldots,\cI_k\subseteq[1,N]$ with $|\cI_i|=Q+O(1)$, where
$$
b_j^*=b_jW+b_0.
$$
Since $h$ is an arbitrary positive integer not greater than $h_0$, it suffices to show that
\begin{equation}\label{phiWkWkQkpsix}
\frac{\phi(W)^{h}}{W^{h}Q^k}\sum_{\vec{x}\in\cI_1\times\cdots\times\cI_k}\prod_{j=1}^h\varpi\big(\psi_j(\vec{x})W+b_j^*\big)=1+o(1).
\end{equation}

Let us turn to the proof of (\ref{phiWkWkQkpsix}). Let
$$
H_1=N^{1-\gamma+\delta_0}.
$$
Recall that in view of (\ref{rhoalphasum}),
$$
\varpi(n)=\frac{\Lambda_R(n)^2}{\log R}\cdot \frac{1}{2\eta_0}\bigg(2\eta_0+\sum_{|s|\geq 1}\alpha_se(sn^\gamma)\bigg).
$$
And by (\ref{fourieralphabound}),
\begin{align*}
\sum_{|s|\geq H_1}|\alpha_s|\ll_{r_0}\sum_{|s|\geq H_1}\frac{1}{\eta_0^{r_0}|s|^{r_0+1}}\ll&\eta_0^{-r_0}H_1^{-r_0+\frac{h+2}{1-\gamma+\delta_0}}\sum_{|s|\geq H_1}|s|^{-\frac{h+2}{1-\gamma+\delta_0}-1}\\
\ll_{r_0}&H_1^{-\frac{h+2}{1-\gamma+\delta_0}}=N^{-h-2}.
\end{align*}
Hence for any intervals $\cI_1,\ldots,\cI_k\subseteq\substack[1,N]$ with $\cI_i=Q+O(1)$, we have
\begin{align*}
&\sum_{\vec{x}\in \vec{\cI}}\prod_{j=1}^h \varpi\big(\psi_{j}(\vec{x})W+b_j^*\big)\\
=&\sum_{\vec{x}\in \vec{\cI}}
\prod_{j=1}^h\frac{\Lambda_R(\psi_{j}(\vec{x})W+b_j^*)^2}{\log R}\bigg(1+\frac{1}{2\eta_0}\sum_{1\leq|s|\leq H_1}\alpha_se\big(s\cdot(\psi_{j}(\vec{x})W+b_j^*)^{\gamma}\big)\bigg)+O(N^{-1}),
\end{align*}
where $\vec{\cI}=\cI_1\times\cdots\times\cI_k$.
Now
\begin{align*}
&\sum_{\vec{x}\in \vec{\cI}}
\prod_{j=1}^h\frac{\Lambda_R(\psi_{j}(\vec{x})W+b_j^*)^2}{\log R}\bigg(1+\frac{1}{2\eta_0}\sum_{1\leq|s|\leq H_1}\alpha_se\big(s\cdot(\psi_{j}(\vec{x})W+b_j^*)^{\gamma}\big)\bigg)\\
=&\sum_{\vec{x}\in \vec{\cI}}\prod_{j=1}^h\frac{\Lambda_R(\psi_{j}(\vec{x})W+b_j^*)^2}{\log R}\sum_{I\subseteq\{1,\ldots,h\}}
\frac{1}{(2\eta_0)^{|I|}}\prod_{i\in I}\bigg(\sum_{1\leq|s|\leq H_1}\alpha_se\big(s\cdot(\psi_{i}(\vec{x})W+b_i^*)^{\gamma}\big)\bigg)\\
=&\sum_{I\subseteq\{1,\ldots,h\}}
\prod_{j=1}^h\frac{\Lambda_R(\psi_{j}(\vec{x})W+b_j^*)^2}{\log R}\sum_{\substack{1\leq|s_i|\leq H_1}}\prod_{l\in I}\frac{\alpha_{s_l}}{2\eta_0}\sum_{\vec{x}\in \vec{\cI}}
e\bigg(\sum_{i\in I}s_i\cdot \big(\psi_{i}(\vec{x})W+b_i^*\big)^{\gamma}\bigg).
\end{align*}

Notice that Green and Tao \cite[Propostion 9.5]{GT08} had proven that
\begin{equation}
\frac{\phi(W)^h}{W^h}\sum_{\vec{x}\in \vec{\cI}}
\prod_{j=1}^h\frac{\Lambda_R(\psi_{j}(\vec{x})W+b_j^*)^2}{\log R}=\big(1+o(1)\big)\prod_{j=1}^h|\cI_j|.
\end{equation}
And we have those $\alpha_{s_l}\ll_{r_0}\eta_0$ by (\ref{fourieralphabound}). Hence it suffices to show that
\begin{equation}
\prod_{j=1}^h\frac{\Lambda_R(\psi_{j}(\vec{x})W+b_j^*)^2}{\log R}\sum_{\vec{x}\in \vec{\cI}}
e\bigg(\sum_{i\in I}s_i\cdot \big(\psi_{i}(\vec{x})W+b_i^*\big)^{\gamma}\bigg)=o\bigg(\frac{Q^{k}}{H_1^{|I|}}\cdot\frac{W^h}{\phi(W)^h}\bigg)
\end{equation}
for any $\emptyset\neq I\subseteq\{1,\ldots,h\}$ and those $s_i$ with $1\leq |s_i|\leq H_1$. Clearly in view of (\ref{LambdaRn}),
\begin{align*}
&\prod_{j=1}^h\frac{\Lambda_R(\psi_{j}(\vec{x})W+b_j^*)^2}{\log R}\sum_{\vec{x}\in \vec{\cI}}
e\bigg(\sum_{i\in I}s_i\cdot \big(\psi_{i}(\vec{x})W+b_i^*\big)^{\gamma}\bigg)\\
=&\frac1{(\log R)^{2h}}\sum_{\substack{d_1,\ldots,d_h,e_1,\ldots,e_h\leq R\\ (d_j,W)=(e_j,W)=1}}\prod_{j=1}^h\mu(d_j)\mu(e_j)\log\frac{R}{d_j}\log\frac{R}{e_j}\\
&\cdot\sum_{\substack{\vec{x}\in\vec{\cI}\\
[d_j,e_j]\mid \psi_j(\vec{x})W+b_j^*\\\text{for each }1\leq j\leq h}}e\bigg(\sum_{i\in I}s_i\cdot\big(\psi_{i}(\vec{x})W+b_i^*\big)^{\gamma}\bigg).
\end{align*}

Fix $1\leq d_1,\ldots,d_h,e_1,\ldots,e_h\leq R$ with $(d_je_j,W)=1$ for $1\leq j\leq h$. Let $D_j=[d_j,e_j]$ for $1\leq j\leq h$ and let $D=[D_1,\ldots,D_h]$.
Let
$$
\vec{\cI}_W=\{(x_1,\ldots,x_k):\; W{\mathfrak a}_i\leq x_i\leq W{\mathfrak b}_i\text{ for }1\leq i\leq k\}
$$
provided that $\vec{\cI}=[{\mathfrak a}_1,{\mathfrak b}_1]\times\cdots[{\mathfrak a}_k,{\mathfrak b}_k]$.
Then since $\psi_i(x_1,\ldots x_k)W=\psi_i(x_1W,\cdots,x_kW)$,
\begin{align*}
&\sum_{\substack{\vec{x}\in\vec{\cI}\\
[d_j,e_j]\mid \psi_j(\vec{x})W+b_j^*\\\text{for any }1\leq j\leq h}}e\bigg(\sum_{i\in I}s_i\cdot\big(\psi_{i}(\vec{x})W+b_i^*\big)^{\gamma}\bigg)\\
=&
\sum_{\substack{\vec{y}=(y_1,\ldots,y_k)\in \vec{\cI}_W\\
\psi_j(\vec{y})+b_j^*\equiv0\pmod{D_j} \\
y_j\equiv0\pmod{W}}}e\bigg(\sum_{i\in I}s_i\cdot\big(\psi_{i}(\vec{y})+b_i^*\big)^{\gamma}\bigg)\\
=&\frac1{W^kD^h}
\sum_{\substack{0\leq u_1,\ldots,u_h<D\\ 0\leq v_1,\ldots,v_k<W}}\sum_{\substack{\vec{y}\in \vec{\cI}_W\\
\vec{y}=(y_1,\ldots,y_k)}}e\bigg(\sum_{i\in I}s_i\big(\psi_{i}(\vec{y})+b_i^*\big)^{\gamma}+\sum_{j=1}^{h}\frac{(\psi_j(\vec{y})+b_j^*)u_j}{D}+\sum_{j=1}^{k}\frac{y_jv_j}{W}\bigg).
\end{align*}
So we only need to show that
\begin{equation}\label{expsum}
\sum_{\substack{\vec{y}\in \vec{\cI}_W\\
\vec{y}=(y_1,\ldots,y_k)}}e\bigg(\sum_{i\in I}s_i\big(\psi_{i}(\vec{y})+b_i^*\big)^{\gamma}+\sum_{j=1}^{h}\frac{(\psi_j(\vec{y})+b_j^*)u_j}{D}+\sum_{j=1}^{k}\frac{y_jv_j}{W}\bigg)=o\bigg(\frac{Q^k}{H_1^{h}R^{2h}}\bigg).
\end{equation}

However, in (\ref{expsum}), it is difficult to give a suitable lower bound for the second derivatives of the sum in $e(\cdot)$, since perhaps some $s_i$ are positive and the other $s_i$ are negative. That is, we can't directly apply the classical van Corput theorem to (\ref{expsum}).
There are two auxiliary lemmas in the next section, which are the key ingredients of our proof.

\section{Two auxiliary lemmas}
\setcounter{lemma}{0}
\setcounter{theorem}{0}
\setcounter{corollary}{0}
\setcounter{equation}{0}
\setcounter{conjecture}{0}
\begin{lemma}\label{matrix}
Let $A=(a_{ij})_{\substack{1\leq i\leq h\\ 1\leq j\leq k}}$ be a $h\times k$ matrix with integral coefficients and let $M=\max_{\substack{1\leq i\leq h\\ 1\leq j\leq k}}|a_{ij}|$. Assume that each two lines of $A$ are linearly independent.
Then there exists a non-singular matrix $T\in\Z^{k\times k}$
such that $|c_{11}|,|c_{21}|,\ldots,|c_{h1}|$ are distinct positive integers bounded by $h^4M$, where
those $c_{ij}$ are given by $$(c_{ij})_{\substack{1\leq i\leq h\\ 1\leq j\leq k}}=(a_{ij})_{\substack{1\leq i\leq h\\ 1\leq j\leq k}}T.$$
\end{lemma}
\begin{proof}
We shall obtain the matrix $(c_{ij})$ via a sequence of column operations on $(a_{ij})$. First, We need to get a new matrix $(b_{ij})_{\substack{1\leq i\leq h\\ 1\leq j\leq k}}$ such that $b_{11},\ldots,b_{h1}$ are all non-zero. Assume that $a_{i_01}=0$. Since $(a_{i_01},\ldots,a_{i_0k})$ is not a zero vector, there exists $2\leq j_0\in k$ such that $a_{i_0j_0}\neq 0$. Note that $|\{1\leq i\leq h:\, a_{i1}\neq 0\}|\leq h-1$ now. By the pigeonhole role, we may choose $\theta\in\{1,2,\ldots,h\}$
such that $$a_{i1}+\theta\cdot a_{ij_0}\neq 0$$ for each $1\leq i\leq h$ with $a_{i1}\neq 0$.
Then the first element of the $i_0$-th line of the new matrix
$$
\left(\begin{matrix}
a_{11}+\theta\cdot a_{1j_0}&a_{12}&\cdots&a_{1k}\\
\vdots&\vdots&\ddots&\vdots\\
a_{h1}+\theta\cdot a_{hj_0}&a_{h2}&\cdots&a_{hk}\\
\end{matrix}\right)
$$
is non-zero. And for each $1\leq i\leq h$ with $a_{i1}\neq 0$, the first element of the $i$-th line of the above new matrix is still non-zero.
Continuing the above process at most $h$ times, we obtain a matrix $(b_{ij})$ with $b_{11},\ldots,b_{h1}\neq 0$ .

Let us turn to the matrix $(b_{ij})$. We shall apply an operation to $(b_{ij})$ such that the number
$$
|\{(s,t):\,1\leq s<t\leq h,\ |b_{s1}|=|b_{t1}|\}|
$$
can be reduced by at least $1$.
Assume that $|b_{s_01}|=|b_{t_01}|\neq 0$ for some distinct $s_0,t_0$.
Since $(b_{s_01},\ldots,b_{s_0k})$ and $(b_{t_01},\ldots,b_{t_0k})$ are also linearly independent, we may choose $2\leq l_0\leq k$ such that
$$
\begin{cases}b_{s_0l_0}\neq b_{t_0l_0},&\text{if }b_{s_01}=b_{t_01},\\
b_{s_0l_0}\neq -b_{t_0l_0},&\text{if }b_{s_01}=-b_{t_01}.\\
\end{cases}
$$
Clearly
$$|b_{s_01}+\theta\cdot b_{s_0l_0}|=|b_{t_01}+\theta\cdot b_{t_0l_0}|$$ implies that
either $$b_{s_01}-b_{t_01}=\theta\cdot(b_{t_0l_0}-b_{s_0l_0}),$$ or
$$b_{s_01}+b_{t_01}=-\theta\cdot(b_{t_0l_0}+b_{s_0l_0}).$$
So there exists at most one $\theta\neq 0$ such that $$|b_{s_01}+\theta\cdot b_{s_0l_0}|=|b_{t_01}+\theta\cdot b_{t_0l_0}|.$$
Hence,
we may choose one $\theta$ such that
$$|b_{s_01}+\theta\cdot b_{s_0l_0}|\cdot|b_{t_01}+\theta\cdot b_{t_0l_0}|>0,\qquad|b_{s_01}+\theta\cdot b_{s_0l_0}|\neq |b_{t_01}+\theta\cdot b_{t_0l_0}|.$$

On the other hand, if $s\neq t$ and $|b_{s1}|\neq|b_{t1}|$, then $$|b_{s1}+\theta\cdot b_{sl_0}|=|b_{t1}+\theta\cdot b_{tl_0}|$$ implies that either
$$b_{s1}-b_{t1}=\theta\cdot(b_{tl_0}-b_{sl_0}),$$ or
$$b_{s1}+b_{t1}=-\theta\cdot(b_{tl_0}+b_{sl_0}).$$
Hence, there exist at most four integers $\theta$ such that either
$$|b_{s1}+\theta\cdot b_{sl_0}|=|b_{t1}+\theta\cdot b_{tl_0}|,$$ or $$b_{s1}+\theta\cdot b_{sl_0}=0,$$ or
$$b_{t1}+\theta\cdot b_{tl_0}=0.$$

Thus by the pigeonhole role, we may choose $\theta\in\{1,2,\ldots,2h(h-1)\}$ such that

\medskip\noindent (i) $|b_{s_01}+\theta\cdot b_{s_0l_0}|$ and $|b_{t_01}+\theta\cdot b_{t_0l_0}|$ are distinct positive integers;

\medskip\noindent (ii) $|b_{s1}+\theta\cdot b_{sl_0}|$ and $|b_{t1}+\theta\cdot b_{tl_0}|$ are distinct positive integers, provided that
$|b_{s1}|$ and $|b_{t1}|$ are  distinct positive integers.

\medskip
Transfer the matrix $(b_{ij})$ to a new matrix
$$
\left(\begin{matrix}
b_{11}+\theta\cdot b_{1l_0}&b_{12}&\cdots&b_{1k}\\
\vdots&\vdots&\ddots&\vdots\\
b_{h1}+\theta\cdot b_{hl_0}&b_{h2}&\cdots&b_{hk}\\
\end{matrix}\right).
$$
Repeating such a process at most $h(h-1)/2$ times, we may obtain the expected new matrix $(c_{ij})$ .
\end{proof}

\begin{lemma}\label{psisum}
Suppose that $M\geq 2$ and let $\psi_i(x)=\alpha_i x+\beta_i$, $1\leq i\leq h$, be some linear functions with $\alpha_i\in\Z$
satisfying that
$$
1\leq |\alpha_{1}|<|\alpha_{2}|<\cdots<|\alpha_{h}|\leq M.
$$
Let $\cI\subseteq\N$ be an interval of integers. Suppose that
$$
N\leq \psi_i(x)\leq \omega N
$$
for any $x\in\cI$ and any $1\leq i\leq h$, where
$$
\omega=1+\frac{1}{2(M-1)}.
$$
Let $\theta_1,\ldots,\theta_h\in\R$ and let
$$
F(x)=\theta_1\psi_1(x)^\gamma+\cdots+\theta_h\psi_h(x)^\gamma.
$$
Then there exists $$1\leq r\leq 3M^2\log M\log h\cdot(1+M\log M)^{h-1}
$$ such that for every $x\in\cI$,
$$
\frac{\omega^{\gamma-r}|\alpha_{h}|^r}{2h^{3(1+M\log M)^h}}\cdot \max_{1\leq i\leq h}\{|\theta_i|\}\leq\frac{|F^{(r)}(x)|}{|(\gamma)_r|\cdot N^{\gamma-r}}\leq \frac{3|\alpha_{h}|^r}{2}\cdot\max_{1\leq i\leq h}\{|\theta_i|\},
$$
where
$$
(\gamma)_r=\gamma(\gamma-1)\cdots(\gamma-r+1).
$$
\end{lemma}
\begin{proof}
First, we shall give some $L_0,L_1,\ldots,L_{h-1}>0$ and $R_1,R_2\ldots,R_{h}\in\N$.
Let $L_{0}=1$ and $R_1=1$.
For $1\leq j\leq h$, assume that $L_0,\ldots,L_{j-1}$ and $R_1,\ldots,R_{j}$ have been given.
Since $\alpha_1,\ldots,\alpha_h$ are all non-zero integers lying in $[-M,M]$,
evidently $|\omega \alpha_{i}|<|\alpha_{i+1}|$ for each $1\leq i<h$.
Let \begin{equation}\label{Lidef}
L_{j}=\frac{2h}{\omega^\gamma}\cdot\bigg(\frac{\omega |\alpha_{h}|}{|\alpha_{j}|}\bigg)^{R_{j}}
\end{equation}
and let $R_{j+1}$ be the least positive integer such that
\begin{equation}\label{Ridef}
\bigg(\frac{|\alpha_{j+1}|}{\omega |\alpha_{j}|}\bigg)^{R_{j+1}}\geq\frac{2h}{\omega^\gamma}\cdot L_1\cdots L_{j}.
\end{equation}

Let
$$
\cJ=\big\{1\leq j<h:\,|\theta_j|\geq L_{j}\cdot\max_{i>j}\{|\theta_i|\}\big\}.
$$
First, suppose that $\cJ$ is non-empty. Let
$j_0$ be the least element of $\cJ$
and let $$r=R_{j_0}.$$
For any $x\in\cI$, evidently
\begin{align}\label{fkxj0}
\frac{|F^{(r)}(x)|}{|(\gamma)_r|}=&\bigg|\sum_{j=1}^h\alpha_{j}^r\theta_j\cdot \psi_j(x)^{\gamma-r}\bigg|\notag\\
\geq&|\alpha_{j_0}|^r|\theta_{j_0}|\cdot \omega^{\gamma-r}N^{\gamma-r}-N^{\gamma-r}\sum_{\substack{1\leq i\leq h\\ i\neq j_0}}|\alpha_{i}|^r|\theta_{i}|,
\end{align}
by noting that $\psi_i(x)\in[N,\omega N]$ and $\gamma-r<0$.

We claim that for each $i\neq j_0$,
\begin{equation}\label{cj0thetax}
|\alpha_{j_0}|^r|\theta_{j_0}|\geq 2\omega^{r-\gamma}h\cdot|\alpha_{i}|^r|\theta_{i}|.
\end{equation}
Since $j_0=\min \cJ$, we have
$$
|\theta_i|<L_i\max_{t>i}|\theta_{t}|
$$
for each $1\leq i<j_0$, i.e.,
$$
\max_{t\geq i}|\theta_i|<L_i\max_{t\geq i+1}|\theta_{t}|.
$$
So if $i<j_0$, then
\begin{align}\label{thetai}
|\theta_i|< L_{i}\max_{t\geq i+1}|\theta_{t}|<
L_{i} L_{i+1}|\max_{t\geq i+2}|\theta_{t}|<\cdots
<&L_{i}L_{i+1}\cdots L_{j_0-1}\max_{t\geq j_0}|\theta_{t}|\notag\\
=&
L_{i}L_{i+1}\cdots L_{j_0-1}|\theta_{j_0}|.
\end{align}
In view of (\ref{Ridef}) and $r=R_{j_0}$,
\begin{align*}
\frac{|\alpha_{j_0}|^r|\theta_{j_0}|}{|\alpha_{i}|^r|\theta_{i}|}\geq&
\frac{|\alpha_{j_0}|^r}{|\alpha_{j_0-1}|^r}\cdot\frac{|\theta_{j_0}|}{|\theta_{i}|}
>2\omega^{r-\gamma}h L_1\cdots L_{j_0-1}
\cdot\frac{1}{L_{i}\cdots L_{j_0-1}}\geq 2\omega^{r-\gamma}h,
\end{align*}
by noting that $L_j\geq 1$ for any $j$.

Suppose that $i>j_0$. Then by (\ref{Lidef}),
$$
\frac{|\alpha_{i}|^r}{|\alpha_{j_0}|^r}\leq \frac{|\alpha_{h}|^r}{|\alpha_{j_0}|^r}=\frac{L_{j_0}}{2\omega^{r-\gamma}h}.
$$
So
\begin{align*}
\frac{|\alpha_{j_0}|^r}{|\alpha_{i}|^r}\cdot\frac{|\theta_{j_0}|}{|\theta_{i}|}\geq
\frac{2\omega^{r-\gamma}h}{L_{j_0}}\cdot L_{j_0}=2\omega^{r-\gamma}h.
\end{align*}
Thus (\ref{cj0thetax}) is always valid.

Combining (\ref{fkxj0}) with (\ref{cj0thetax}), we get
$$
\frac{|F^{(r)}(x)|}{|(\gamma)_r|}\geq\frac{\omega^{\gamma-r}|\alpha_{j_0}|^r|\theta_{j_0}|}{2}\cdot N^{\gamma-r}.
$$
On the other hand, by (\ref{thetai}), clearly we have
\begin{align*}
|\theta_{j_0}|\geq\frac{1}{L_1\cdots L_{h-1}}\cdot\max_{1\leq i\leq h}\{|\theta_{i}|\}.
\end{align*}
Hence
$$
\frac{|F^{(r)}(x)|}{|(\gamma)_r|}\leq|\alpha_{j_0}|^r|\theta_{j_0}|\cdot N^{\gamma-r}+
N^{\gamma-r}\sum_{\substack{1\leq i\leq h\\ i\neq j_0}}|\alpha_{i}^r\theta_{i}|\leq \frac{3|\alpha_{j_0}|^r|\theta_{j_0}|}{2}\cdot N^{\gamma-r}.
$$

Next, suppose that $\cJ$ is empty. For any $1\leq i\leq h-1$, in view of (\ref{thetai}), similarly we have
$$
|\theta_i|<L_{i}\max_{t\geq i+1}|\theta_{t}|<L_{i}L_{i+1}\max_{t\geq i+2}|\theta_{t}|<
\cdots
<L_{i}L_{i+1}\cdots L_{h-1}|\theta_{h}|.
$$
Letting $r=R_h$, we get
\begin{align*}
\frac{|\alpha_{h}^r\cdot\theta_{h}|}{|\alpha_{i}^r\cdot\theta_{i}|}\geq
\frac{|\alpha_{h}|^r}{|\alpha_{h-1}|^r}\cdot\frac{|\theta_{h}|}{|\theta_{i}|}\geq
2\omega^{r-\gamma}h L_1\cdots L_{h-1}
\cdot\frac{1}{L_{i}\cdots L_{h-1}}\geq 2\omega^{r-\gamma}h.
\end{align*}
It follows that
\begin{align*}
\frac{|F^{(r)}(x)|}{|(\gamma)_r|}\geq&|\alpha_{h}|^r|\theta_{h}|\cdot \omega^{\gamma-r}N^{\gamma-r}-N^{\gamma-r}\sum_{i=1}^{h-1}|\alpha_{i}^r\theta_{i}|\\
\geq&\frac{|\alpha_{h}|^r|\theta_{h}|}2\cdot \omega^{\gamma-r}N^{\gamma-r}\geq\frac{|\alpha_{h}|^r\omega^{\gamma-r}}{2L_1\cdots L_{h-1}}\cdot\max_{1\leq i\leq h}\{|\theta_{i}|\}\cdot N^{\gamma-r},
\end{align*}
and
$$
\frac{|F^{(r)}(x)|}{|(\gamma)_r|}\leq|\alpha_{h}|^r|\theta_{h}|\cdot N^{\gamma-r}+
N^{\gamma-r}\sum_{i=1}^{h-1}|\alpha_{i}^r\theta_{i}|\leq \frac{3|\alpha_{h}|^r|\theta_{h}|}2\cdot N^{\gamma-r}.
$$

Finally, we need to give a upper bound for $R_h$. Clearly
\begin{align}\label{logLj}
\log L_{j}=&\log(2h\omega^{-\gamma})+R_j\log\frac{\omega|\alpha_h|}{|\alpha_j|}\notag\\
\leq&\log(2h\omega^{-\gamma})+\log\frac{\omega|\alpha_h|}{|\alpha_j|}\cdot\bigg(\log\frac{\omega|\alpha_j|}{|\alpha_{j-1}|}\bigg)^{-1}\cdot\bigg(\log(3h\omega^{-\gamma})+\sum_{i=1}^{j-1}\log L_i\bigg)\notag\\
\leq&\log(2h\omega^{-\gamma})+\log M\cdot M\cdot\bigg(\log(3h\omega^{-\gamma})+\sum_{i=1}^{j-1}\log L_i\bigg).
\end{align}
We claim that
\begin{equation}\label{ublogLj}
\log L_j\leq 3M\log M\log h\cdot(1+M\log M)^j
\end{equation}
for each $0\leq j\leq h$.
In fact, assume that (\ref{ublogLj}) holds for $L_1,\ldots,L_{j-1}$. Then by (\ref{logLj}),
\begin{align*}
\log L_{j}\leq&\log(2h\omega^{-\gamma})+M\log M\bigg(\log(3h\omega^{-\gamma})+3M\log M\log h\cdot\frac{(1+M\log M)^j-1}{M\log M}\bigg)\\
\leq& M\log M\cdot \log h(1+M\log M)^j,
\end{align*}
since it is easy to verify
$$\log(2h\omega^{-\gamma})+M\log M\log(3h\omega^{-\gamma})\leq 3M\log M\log h.
$$
So
$$
R_h\leq \log L_{h-1}\cdot\bigg(\log\frac{\omega|\alpha_h|}{|\alpha_{h-1}|}\bigg)^{-1}\leq 3M^2\log M\log h\cdot(1+M\log M)^{h-1},
$$
and
$$
\log L_1+\cdots+\log L_{h-1}\leq 3\log h\cdot(1+M\log M)^h.
$$
\end{proof}

\section{Proof of Theorem \ref{main}}
\setcounter{lemma}{0}
\setcounter{theorem}{0}
\setcounter{corollary}{0}
\setcounter{equation}{0}
\setcounter{conjecture}{0}

In this section, we shall complete the proof Theorem \ref{main}.
Let $\nu$ be the pseudorandom measure constructed in (\ref{nufunc}). According to Theorem \ref{CFZThm}, we only need to verify that $\nu$ obeys the $(2^{m-1}m,2m,m)$-linear forms condition.

Recall that $h_0=2^{m-1}m$ and $k_0=2m$. Suppose that $1\leq h\leq h_0$ and $1\leq k\leq k_0$. Suppose that $$\psi_i(\vec{x})=a_{i1}x_1+\cdots+a_{ik_0}x_{k},\qquad 1\leq i\leq h$$ with $|a_{ij}|\leq m$, and $b_1,\ldots,b_i\in\Z_N$. As we has mentioned, it suffices to show that
\begin{equation} \label{linearconditioneb}
\frac1{N^k}\sum_{\vec{x}\in\Z_N^k}\prod_{i=1}^h\nu\big(\psi_i(\vec{x})+b_i\big)=1+o(1).
\end{equation}
By Lemma \ref{matrix}, there exists a non-singular matrix $T\in\Z_N^{k\times k}$ such that
$(a_{ij}^*)_{\substack{1\leq i\leq h\\ 1\leq j\leq k}}=(a_{ij})_{\substack{1\leq i\leq h\\ 1\leq j\leq k}}T$ satisfies
$1\leq |a_{11}^*|<\ldots<|a_{h1}^*|\leq h^4m$.
Since $T$ is non-singular,
\begin{align*}
\frac1{N^k}\sum_{\vec{x}\in\Z_N^k}\prod_{i=1}^h\nu\big(\psi_i(\vec{x})\big)=&\frac1{N^k}\sum_{\vec{x}\in\Z_N^k}\prod_{i=1}^h\nu\big(\psi_i(\vec{x}T)\big)\\
=&\frac1{N^k}\sum_{\vec{x}\in\Z_N^k}\prod_{i=1}^h\nu(a_{i1}^*x_1+\cdots+a_{ik_0}^*x_{k}+b_i).
\end{align*}
Therefore without loss of generality, we may assume that those linear functions $\psi_i(\vec{x})=a_{i1}x_1+\cdots+a_{ik_0}x_{k}$ in (\ref{linearconditioneb}) satisfy
$$
1\leq |a_{11}|<|a_{21}|<\ldots<|a_{h1}|\leq h^4m.
$$

According to our discussions in Section 3, (\ref{linearconditioneb}) follows from (\ref{expsum}), i.e.,
\begin{equation}\label{expsumb}
\sum_{\substack{\vec{x}\in \vec{\cI}\\
\vec{x}=(x_1,\ldots,x_k)}}e\bigg(\sum_{i\in I}s_i\big(\psi_{i}(\vec{x})+b_i^*\big)^{\gamma}+\sum_{j=1}^{h}\frac{(\psi_j(\vec{x})+b_j^*)u_j}{D}+\sum_{j=1}^{k}\frac{x_jv_j}{W}\bigg)=o(Q^k\cdot H_1^{-h}R^{-2h}),
\end{equation}
where $$1\leq |s_i|\leq H_1,\qquad 1\leq u_j\leq D,\qquad 1\leq v_j\leq W$$
and $\cI=\cI_1\times\cdots\times \cI_k$ with $$|\cI_1|,\ldots,|\cI_k|=WQ+O(W).$$
Below we need the following general form of van der Corput's theorem \cite[Satz 4]{C29}:
\begin{lemma} Let $f(x)$ be a smooth function on the interval $[X,X+Y]$. Suppose that $r\geq 2$ and
$$
0<\lambda\leq|f^{(r)}(x)|\leq \alpha\lambda
$$
on $[X,X+Y]$. Then
\begin{equation}\label{VC}
\sum_{X\leq n\leq X+Y}e\big(f(n)\big)\ll
\alpha Y\big(\lambda^{\frac1{2^r-2}}+Y^{-\frac1{2^{r-1}}}+(Y^r\lambda)^{-\frac1{2^{r-1}}}\big).
\end{equation}
\end{lemma}
Let
$$
F_{x_2,\ldots,x_k}(y)=\sum_{i\in I}s_i\big(\psi_{i}(y,x_2,\ldots,x_k)+b_i^*\big)^{\gamma}+\sum_{j=1}^{h}\frac{(\psi_j(y,x_2,\ldots,x_k)+b_j^*)u_j}{D}+\sum_{j=1}^{k}\frac{y_jv_j}{W}.
$$
Clearly
$$
\bigg|\sum_{\substack{\vec{x}\in \vec{\cI}\\
\vec{x}=(x_1,\ldots,x_k)}}e\big(F_{x_2,\ldots,x_k}(x_1)\big)\bigg|\leq
\sum_{\substack{x_2\in\cI_2,\ldots,x_k\in\cI_k}}\bigg|\sum_{y\in\cI_1}e\big(F_{x_2,\ldots,x_k}(y)\big)\bigg|.
$$
Let $$M_0=h_0^4m.$$
Applying Lemma \ref{psisum} to $F_{x_2,\ldots,x_k}'$, for any given $x_2\in\cI_2,\ldots,x_k\in\cI_k$, there exists
$$
2\leq r\leq 3M_0^2\log M_0\log h_0\cdot(1+M_0\log M_0)^{h_0-1}+1,
$$
such that for any $y\in\cI_1$
$$
c_1\Psi N^{\gamma-r}\leq |F_{x_2,\ldots,x_k}^{(r)}(y)|\leq c_2\Psi N^{\gamma-r},
$$
where $c_1,c_2>0$ are two constants only depending on $m$ and
$$
\Psi=\max_{i\in I}|s_i|.
$$
Let $\lambda=\Psi N^{\gamma-r}$.  Since $\Psi\leq H_1=N^{1-\gamma+\delta_0}$,
$$
\lambda^{\frac{1}{2^r-2}}\leq (H_1N^{\gamma-r})^{\frac{1}{2^r-2}}=N^{\frac{1+\delta_0-r}{2^r-2}}.
$$
Recalling that  $W\gg\log_3N$ and $Q=[N/\log_4N]$,
Let $Y=WQ$
we have
and
$$
(Y^r\lambda )^{-\frac1{2^{r-1}}}\leq
(N^r\cdot N^{\gamma-r} )^{-\frac1{2^{r-1}}}
= N^{-\frac{\gamma}{2^{r-1}}}.
$$
Using Lemma \ref{VC}, we get that
\begin{align*}
\sum_{y\in\cI_1}e\big(F_{x_2,\ldots,x_k}(y)\big)\ll&Y\big(\lambda^{\frac1{2^r-2}}+Y^{-\frac1{2^{r-1}}}+(Y^r\lambda)^{-\frac1{2^{r-1}}}\big)\\
\ll&WQ\cdot N^{-2^{-2-3M_0^2\log M_0\log h_0\cdot(1+M_0\log M_0)^{h_0-1}}}.
\end{align*}
It is not difficult to check that
$$
3M_0^2\log M_0\log h_0\cdot(1+M_0\log M_0)^{h_0-1}+3\leq2^{4^mm^2}-\frac{\log h_0}{\log 2}
$$
for each $m\geq3$.
Hence recalling that $1-r+3\delta_0=2^{-2^{4^mm^2}}$, we have
\begin{align*}
\sum_{\substack{\vec{x}\in \vec{\cI}\\
\vec{x}=(x_1,\ldots,x_k)}}e\big(F_{x_2,\ldots,x_k}(x_1)\big)\ll&
W^kQ^k\cdot N^{-h_02^{1-2^{4^mm^2}}}\\
\ll& Q^k\cdot N^{-h_0(1-\gamma+3\delta_0)}=Q^k\cdot H_1^{-h_0}R^{-2h_0}.
\end{align*}
Thus (\ref{expsumb}) is concluded, i.e., the function $\nu$ really obeys the $(2^{m-1}m,2m,m)$-linear forms condition.

\end{document}